\documentclass[11pt]{amsart}
\usepackage[margin=1.12in]{geometry}
\usepackage{amsmath,amssymb,mathtools,booktabs}
\usepackage{xcolor}
\usepackage{hyperref}
\hypersetup{colorlinks=true,linkcolor=blue!50!black,citecolor=blue!50!black,urlcolor=blue!50!black,
pdftitle={Punctured adjacency-degree algebras of Cartesian products},
pdfauthor={Zhipeng Lu}}
\usepackage{aliascnt}
\usepackage[nameinlink,capitalise]{cleveref}
\newtheorem{theorem}{Theorem}[section]
\newaliascnt{proposition}{theorem}
\newtheorem{proposition}[proposition]{Proposition}
\aliascntresetthe{proposition}
\crefname{proposition}{Proposition}{Propositions}
\newaliascnt{lemma}{theorem}
\newtheorem{lemma}[lemma]{Lemma}
\aliascntresetthe{lemma}
\crefname{lemma}{Lemma}{Lemmas}
\newaliascnt{corollary}{theorem}
\newtheorem{corollary}[corollary]{Corollary}
\aliascntresetthe{corollary}
\crefname{corollary}{Corollary}{Corollaries}
\theoremstyle{definition}
\newaliascnt{definition}{theorem}
\newtheorem{definition}[definition]{Definition}
\aliascntresetthe{definition}
\crefname{definition}{Definition}{Definitions}
\newaliascnt{example}{theorem}
\newtheorem{example}[example]{Example}
\aliascntresetthe{example}
\crefname{example}{Example}{Examples}
\theoremstyle{remark}
\newaliascnt{remark}{theorem}
\newtheorem{remark}[remark]{Remark}
\aliascntresetthe{remark}
\crefname{remark}{Remark}{Remarks}
\newcommand{\R}{\mathbb R}
\newcommand{\one}{\mathbf1}
\newcommand{\Balg}{\mathcal B}
\newcommand{\Talg}{\mathcal T}
\newcommand{\Span}{\operatorname{span}}
\newcommand{\Aut}{\operatorname{Aut}}
\newcommand{\End}{\operatorname{End}}
\newcommand{\rank}{\operatorname{rank}}
\newcommand{\diag}{\operatorname{diag}}
\newcommand{\tr}{\operatorname{tr}}
\newcommand{\Spec}{\operatorname{Spec}}

\newcommand{\diam}{\operatorname{diam}}
\newcommand{\ecc}{\operatorname{ecc}}

\newcommand{\Sh}{\mathrm{Sh}}
\newcommand{\ip}[2]{\langle #1,#2\rangle}

\newcommand{\norm}[1]{\lVert #1\rVert}
\newcommand{\bdy}{\mathcal L}
\newcommand{\calC}{\mathcal C}
\newcommand{\defect}{\delta_{\ominus}}

\title[Punctured algebras of Cartesian products]{Punctured adjacency--degree algebras of Cartesian products}
\author{Zhipeng Lu}
\address{Shenzhen MSU--BIT University and Guangdong Laboratory of Machine Perception and Intelligent Computing, Shenzhen, Guangdong 518172, China}
\email{zhipeng.lu@hotmail.com}
\date{}
\subjclass[2020]{05C50, 05E30, 05C60, 15A18}
\keywords{adjacency--degree algebra, vertex deletion, Cartesian product, local spectral measure, Terwilliger algebra, distance-regular graph}

\begin{document}
\begin{abstract}
For a connected regular graph $G$ and a vertex $a$, we study the algebra
generated by the adjacency and degree matrices of $G-a$ and its cyclic
module $P_a$ generated by the all-ones vector.  Our main theorem determines
$\dim P_a$ for Cartesian products whose factors have equitable distance
partitions at the chosen roots.  A normalized logarithmic derivative of
the local spectral generating function partitions the factors into
boundary classes.  We identify the boundary-return space exactly and
express $\dim P_a$ as a sum of affine ranks on additive spectral fibres.
For a distance-regular factor with distinct spectrum $\Theta$, this gives
$\dim P_a(F^{\square m})=|m\Theta|-1$.  For products of powers of two
distinct complete graphs, we evaluate the fibre formula in closed form.
We also determine the full punctured algebras of all Hamming graphs.
Equality with the compressed Terwilliger algebra holds precisely in
dimensions at most four for the hypercube and at most two for larger
alphabets.
For distance-regular graphs, adjacency moments alone determine the
intersection array, with an explicit finite reconstruction.  Finally,
Cartesian stabilizer formulas separate metric loss from orbit splitting;
on Doob graphs their distance-graded defect recovers the number of
Shrikhande factors.
\end{abstract}
\maketitle

\section{Introduction}

The adjacency--degree algebra of a graph is the real unital algebra
generated by its adjacency matrix and diagonal degree matrix.  On a
regular graph the degree matrix is scalar, and the cyclic module generated
by the all-ones vector is one-dimensional.  Deleting a vertex removes this
degeneracy in a controlled way: the remaining degrees distinguish the
neighbors of the deleted vertex.  This leads to a pointed algebra with
only two symmetric generators, rather than a separate diagonal generator
for every distance from the root.

Let $G$ be a finite, simple, connected $d$-regular graph with at least two
vertices, and fix $a\in V(G)$.  Throughout the paper, put
\begin{equation}\label{eq:definitions}
 B=A(G-a),\qquad
 E=\diag(\one_{N(a)}),\qquad
 b=\one_{N(a)},\qquad
 \Balg_a=\langle I,B,E\rangle,\qquad P_a=\Balg_a\one.
\end{equation}
These definitions are equivalent to using $B$ and $D(G-a)$, since
$D(G-a)=dI-E$.  We call $\Balg_a$ the \emph{punctured adjacency--degree
algebra} and $P_a$ its \emph{principal module}.  All algebras and vector
spaces in this paper are over $\R$.  Empty direct sums are omitted.

There are two different questions.  The first asks how much of the
rooted distance structure is generated from $\one$.  The second asks how
much of the entire rooted matrix algebra is generated by $B$ and $E$.
These questions should not be conflated: the principal module can have
dimension equal to the diameter while the full punctured algebra has
many additional irreducible constituents.

\subsection{The product theorem}
For a rooted regular graph $(F,r)$, let $\Theta(F,r)$ be its local spectral
support, and write
\begin{equation}\label{eq:eta-intro}
 \omega(\theta)=\norm{E_\theta e_r}^{2},\qquad
 \Phi(z)=\sum_{\theta\in\Theta(F,r)}\omega(\theta)e^{z\theta},
 \qquad \eta(z)=\frac{1}{d}\frac{\Phi'(z)}{\Phi(z)}.
\end{equation}
Here $E_\theta$ is the adjacency spectral projection and $d$ is the
valency.  The function $\Phi$ is positive on the real line, so $\eta$ is
well-defined there.

Our main result, \cref{thm:main}, applies to a Cartesian product
$G=F_1\square\cdots\square F_\ell$ whenever the distance partition of each
factor is equitable at the specified root.  Distance-regular factors are
an important special case, but neither vertex-transitivity nor Cartesian
primality is needed.  Partition the factor indices into classes
$I_1,\ldots,I_t$ according to equality of the functions $\eta_i$ in
\eqref{eq:eta-intro}.  Set
\[
 T_g=\sum_{i\in I_g}\Theta(F_i,r_i),\qquad
 \Lambda=T_1+\cdots+T_t,
\]
where sums are sums of finite subsets of $\R$.  For $\lambda\in\Lambda$,
form a matrix $R_\lambda$ with one row
\[
 (1,y_1,\ldots,y_t)
 \quad\text{for each }(y_1,\ldots,y_t)\in T_1\times\cdots\times T_t
 \text{ satisfying }y_1+\cdots+y_t=\lambda.
\]
Then
\begin{equation}\label{eq:main-intro}
 \dim P_a(G)=\sum_{\lambda\in\Lambda}\rank R_\lambda-1.
\end{equation}
The substantive step is the exact classification of the boundary-return
space, not the final finite-dimensional spectral rank calculation.  In
particular, the local spectral weights enter through the partition into
boundary classes; the supports alone need not specify that partition.

For a single distance-regular factor $F$ with distinct spectrum $\Theta$,
all copies belong to one class, and \eqref{eq:main-intro} becomes
\[
 \dim P_a(F^{\square m})=|m\Theta|-1.
\]
For $m\ge2$, this equals $m\diam(F)$ exactly when $\Theta$ is an arithmetic
progression.  Different boundary classes can resolve spectral collisions
that the adjacency Krylov space cannot resolve.  We quantify this
phenomenon for $K_p^{\square m}\square K_q^{\square n}$, $p\ne q$.  If
$g=\gcd(p,q)$, $p'=p/g$, and $q'=q/g$, then
\begin{equation}\label{eq:complete-intro}
 \dim P_a=(m+1)(n+1)
       -(m-2q'+1)_+(n-2p'+1)_+-1,
\end{equation}
where $x_+=\max\{x,0\}$; see \cref{thm:complete}.

\subsection{Full algebras, moments, and stabilizers}
For a distance-regular graph of diameter $D$, the principal module is the
span of the $D$ nonzero-distance sphere indicators, and the punctured
algebra acts on it as $M_D(\R)$.  We show that the scalar adjacency moments
$\one^{\mathsf T}B^j\one$, without degree insertions, determine the
intersection array.  Moments with $0\le j\le2L-1$ suffice among graphs of
diameter at most $L$; the proof gives an explicit reconstruction.

Let $\Talg(a)$ be the Terwilliger algebra and $e=I-E_0^*$ the projection
away from the root.  A general corner-generation lemma shows that, when
$G-a$ is connected, the passage from $\Balg_a$ to $e\Talg(a)e$ consists
exactly of adjoining the remaining distance projectors.  Consequently,
the two algebras coincide for distance-regular graphs of diameter at most
two.  This equality does not persist in high diameter.  For the hypercube
$Q_d$, we prove
\begin{equation}\label{eq:cube-intro}
 \Balg_a(Q_d)\cong
 M_d(\R)\oplus M_{d-1}(\R)\oplus\R^{\max\{d-3,0\}}
 \qquad(d\ge2),
\end{equation}
and $\Balg_a(Q_1)=\R$.  It follows that
$\Balg_a(Q_d)=e\Talg(a)e$ if and only if $1\le d\le4$.
For the other Hamming graphs, \cref{thm:hamming} gives
\[
 \Balg_a(H(d,q))\cong M_d(\R)\oplus M_d(\R)\oplus
 M_{d-1}(\R)\oplus\R^{d-1}\qquad(q\ge3,\ d\ge2).
\]
Its dimension is $3d^2-d$, independent of $q\ge3$, and equality with
the compressed corner holds exactly when $d\le2$.

The comparison with symmetry uses
\[
 U_a^-=(\R^{V(G)\setminus\{a\}})^{\Aut(G)_a},\qquad
 \defect(G,a)=\dim U_a^--\dim P_a.
\]
We give the Cartesian product formula for $U_a^-$ and a distance-graded
version.  For the Doob graph
$\Sh^{\square m}\square K_4^{\square n}$, of diameter $N=2m+n$, they yield
\[
 \defect=(n+1)\binom{m+3}{3}-1-N,
\]
and the coefficient in degree $N$ of the graded defect is $m$.  This
last invariant uses stabilizer orbits as well as the punctured algebra;
it is not asserted to be recoverable from adjacency--degree moments.

\subsection{Relation to earlier work}
The unpunctured adjacency--degree algebra and its principal module are
studied by Lu and Li~\cite{LuLi2026}.  We use that terminology, but prove
all structural statements needed here.  The spectral support of a rooted
cyclic module belongs to local spectral graph theory; relevant treatments
include Fiol and Garriga~\cite{FiolGarriga1997} and the walk-matrix viewpoint
of Godsil~\cite{Godsil2012}.  For distance-regular graphs, our sphere
module is the root-deleted primary module of the Terwilliger
algebra~\cite{Terwilliger1992,BCN1989}.  Its tridiagonal action is
classical.  The moment reconstruction below is an application of that
action to the punctured all-ones seed, rather than a new classification
of distance-regular graphs.

Hanaki and Yoshikawa~\cite{HanakiYoshikawa2023} compare algebras
generated by adjacency together with the root projection, the distance
projections, or the stabilizer-orbit projections, and with the stabilizer
centralizer.  Their unpunctured algebra $\langle A,E_0^*\rangle$ is
different from our residual algebra $\langle B,E_1^*|_e\rangle$.  Their comparison supplies the broader framework for the
more specific generation questions addressed here.

The full Terwilliger algebra of a strongly regular graph was studied by
Tomiyama and Yamazaki~\cite{TomiyamaYamazaki1994}, and that of the hypercube
by Go~\cite{Go2002}.  Levstein, Maldonado, and
Penazzi~\cite{LevsteinMaldonadoPenazzi2006} determine the full Terwilliger
algebra of every Hamming scheme.  Our hypercube proof recalls the classical Boolean
lattice decomposition and determines the smaller algebra generated after
deletion by adjacency and the first-sphere projection; the one-coordinate
Hamming decomposition then gives the larger-alphabet case.  Terwilliger and
\v Zitnik~\cite{TerwilligerZitnik2019} compare the subconstituent algebra
with the quantum adjacency algebra.  Their generators distinguish
raising, lowering, and flat parts at all distances.  They are different
from the two residual generators considered here.

The structure of Hamming and Doob graphs, and their intersection arrays,
is classical~\cite{Shrikhande1959,Doob1972,Egawa1981,BCN1989}.  Terwilliger
modules of Doob schemes were determined by Tanabe~\cite{Tanabe1997};
Morales and Palma~\cite{MoralesPalma2021} study the corresponding quantum
adjacency algebras.  We use these graph families to compare the principal
module with the stabilizer orbit module, not to claim new full
Terwilliger decompositions.  Our product orbit formulas use the standard
automorphism theorem for Cartesian prime factors~\cite{HIK2011}.
The additive equality case used for powers is the elementary
one-dimensional small-sumset theorem; see Nathanson~\cite{Nathanson1996}.

The main result of this paper is the boundary classification and the
resulting rank formula for mixed products.  Its concrete consequences
\eqref{eq:complete-intro} and \eqref{eq:cube-intro} concern, respectively,
loss within the principal module and loss in the full compressed algebra.
These distinguish two obstructions that a dimension calculation for a
single distance-regular graph does not see.

\section{Vertex deletion and boundary returns}\label{sec:deletion}

We equip every vertex space with its usual Euclidean inner product.  A
unital algebra generated by symmetric matrices is closed under transpose;
its invariant subspaces are therefore reducing.  We identify a vector on
$G-a$ with its extension by zero at $a$ whenever convenient.

\begin{lemma}\label{lem:automorphisms}
For a connected regular graph $G$ with at least two vertices, restriction
induces $\Aut(G)_a\cong\Aut(G-a)$.  Moreover $P_a\subseteq U_a^-$.
\end{lemma}
\begin{proof}
A vertex of $G-a$ belongs to $N(a)$ exactly when its residual degree is
$d-1$; every other residual degree is $d$.  Thus every automorphism of
$G-a$ preserves $N(a)$ and extends uniquely by fixing $a$.  Conversely,
root-fixing automorphisms restrict to automorphisms of the deletion.
Their permutation matrices commute with $B$ and $E$ and fix $\one$,
which proves the inclusion.
\end{proof}

\begin{proposition}[Spectral deletion]\label{prop:deletion}
Let $\Theta_a=\{\theta:E_\theta e_a\ne0\}$ be the local adjacency spectral
support of $G$ at $a$.  Then
\begin{equation}\label{eq:deletion}
 \R[B]\one=\R[B]b,\qquad
 \dim\R[B]\one=|\Theta_a|-1\ge\ecc(a).
\end{equation}
In particular $P_a$ is the smallest $B,E$-invariant subspace containing $b$.
\end{proposition}
\begin{proof}
The matrix $dI-B$ is positive definite.  Indeed, extend a residual vector
$x$ by $x_a=0$; then
\[
 x^{\mathsf T}(dI-B)x=\sum_{uv\in E(G)}(x_u-x_v)^2,
\]
which vanishes only for the zero vector, since $G$ is connected.  Also
$(dI-B)\one=b$.  The inverse of $dI-B$ is a polynomial in $B$, by spectral
interpolation.  This proves the equality of cyclic spaces and the last
assertion.

Relative to $\R e_a\oplus\R^{V(G-a)}$, the adjacency matrix is
\begin{equation}\label{eq:adjacency-block}
 A=\begin{pmatrix}0&b^{\mathsf T}\\ b&B\end{pmatrix}.
\end{equation}
The subspace $\R e_a\oplus\R[B]b$ is $A$-invariant and contains $e_a$.
Conversely, $Ae_a=b$, and induction in \eqref{eq:adjacency-block} shows
that each $B^kb$ belongs to $\R[A]e_a$.  Hence
\[
 \R[A]e_a=\R e_a\oplus\R[B]b.
\]
The mutually orthogonal nonzero vectors $E_\theta e_a$ form a basis of
$\R[A]e_a$, proving the dimension formula.  Finally,
$e_a,Ae_a,\ldots,A^re_a$ are independent for $r=\ecc(a)$: evaluation at a
vertex at distance $r$, then at distance $r-1$, and so on eliminates
successively the coefficients of a linear relation.  At distance $j$,
$A^je_a$ is positive whereas all earlier powers vanish.
\end{proof}

\begin{definition}
Let $V\subseteq\R^{V(G-a)}$ be $B,E$-invariant and contain $\one$.
Write $\bdy=EV$ and
\[
 H_k=EB^kE\big|_{\bdy}\qquad(k\ge0).
\]
The \emph{boundary-return space} $\calC\subseteq\bdy$ is the smallest
subspace containing $b$ and invariant under all $H_k$.
\end{definition}

\begin{lemma}[Boundary transfer]\label{lem:transfer}
With this notation,
\begin{equation}\label{eq:transfer}
 P_a=\sum_{k\ge0}B^k\calC,
 \qquad
 \dim P_a=\sum_{\lambda\in\Spec(B|_V)}\dim(Q_\lambda\calC),
\end{equation}
where $Q_\lambda$ denotes the spectral projection of $B|_V$.
Only $0\le k<\dim V$ is needed in either the invariance condition or the
span in \eqref{eq:transfer}.
\end{lemma}
\begin{proof}
The space on the right of the first equality contains $b$, is $B$-invariant,
and is $E$-invariant because
$EB^kc=EB^kEc=H_kc\in\calC$.  It therefore contains $P_a$ by
\cref{prop:deletion}.  In the reverse direction, $EP_a$ contains $b$ and
is invariant under the $H_k$, so $\calC\subseteq EP_a\subseteq P_a$.
This proves the first equality.  Spectral interpolation gives
$\R[B]\calC=\bigoplus_\lambda Q_\lambda\calC$, proving the second.
Cayley--Hamilton supplies both finite bounds.
\end{proof}

The return operators can be computed before deletion.  The following
lemma records precisely why this does not change their cyclic subspace.

\begin{lemma}[Full and residual returns]\label{lem:full-returns}
Let $\widetilde V=\R r\oplus V$ be a real Hilbert space, with $\norm r=1$,
and let
\[
 T=\begin{pmatrix}0&b^{\mathsf T}\\b&B\end{pmatrix}
\]
be self-adjoint.  Suppose $\bdy\subseteq V$ contains $b$, and let $\pi$ be
the orthogonal projection onto $\bdy$.  A subspace $C\subseteq\bdy$
containing $b$ is invariant under every $\pi B^k\pi$ if and only if it is
invariant under every $\pi T^k\pi$.
\end{lemma}
\begin{proof}
For real $u$ sufficiently close to zero, put
\[
 F(u)=\pi(I-uB)^{-1}\pi\big|_{\bdy},\qquad
 M(u)=\pi(I-uT)^{-1}\pi\big|_{\bdy},\qquad
 \beta=b^{\mathsf T}F(u)b.
\]
Taking the Schur complement of the root entry, followed by the rank-one
inverse formula, gives
\begin{equation}\label{eq:return-resolvent}
 M=F+\frac{u^2}{1-u^2\beta}(Fb)(Fb)^{\mathsf T}.
\end{equation}
All denominators are nonzero near zero.  In particular,
$Mb=Fb/(1-u^2\beta)$ and
$\gamma:=b^{\mathsf T}Mb=\beta/(1-u^2\beta)$, so equivalently
\begin{equation}\label{eq:return-inverse}
 F=M-\frac{u^2}{1+u^2\gamma}(Mb)(Mb)^{\mathsf T}.
\end{equation}
If $C$ contains $b$ and is invariant under every coefficient of $F$, then
$F(u)C\subseteq C$ and $F(u)b\in C$.  Equation
\eqref{eq:return-resolvent} gives $M(u)C\subseteq C$, and taking Taylor
coefficients gives invariance under all full returns.  The converse
follows from \eqref{eq:return-inverse} in exactly the same way.  Taylor
coefficients may be taken after pairing with any vector in $C^\perp$,
so no convergence or closure issue is involved.
\end{proof}

\section{Local spectral classification of Cartesian products}\label{sec:products}

\subsection{Equitable rooted factors and their spectral model}
A partition of a graph's vertex set is \emph{equitable} if the number of
neighbors in a given cell depends only on the cell containing the vertex.
We shall assume that the distance partition
\[
 S_0(r),S_1(r),\ldots,S_D(r),\qquad D=\ecc(r),
\]
is equitable at the specified root.  This is a condition on the rooted
graph; no assertion about other roots is implicit.

\begin{lemma}\label{lem:rooted-radial}
Suppose $F$ is connected and regular and its distance partition at $r$ is
equitable.  Then
\[
 W(F,r):=\R[A(F)]e_r
 =\Span\{\one_{S_j(r)}:0\le j\le D\}.
\]
Its dimension is $D+1$, and the local spectral support has $D+1$ points.
Under a unitary identification with $L^2(\Theta,\omega)$, $e_r$ corresponds
to the constant function $1$ and $A(F)$ to multiplication by $x$.  Moreover,
\begin{equation}\label{eq:local-moments}
 \mathbb E X=0,\qquad \mathbb E X^2=d,\qquad
 \Phi(0)=1,\quad \eta(0)=0,\quad\eta'(0)=1.
\end{equation}
\end{lemma}
\begin{proof}
Equitability makes the sphere span invariant under adjacency.  In the
orthonormal sphere basis, adjacency is a real symmetric tridiagonal
matrix with positive consecutive off-diagonal entries: every sphere has
edges to the preceding sphere, and equitability makes the corresponding
intersection counts positive.  Its first basis vector is cyclic, since
the successive powers have successively nonzero leading coordinates.
This proves the first assertions.  Map $p(A)e_r$ to the function
$p(x)$ on the local support.  The identity
\[
 \ip{p(A)e_r}{q(A)e_r}
   =\sum_{\theta\in\Theta}\omega(\theta)p(\theta)q(\theta)
\]
proves the asserted unitary identification.  Every weight is positive
by definition.  The first two moments are $(A)_{rr}=0$ and $(A^2)_{rr}=d$,
from which the identities for $\Phi$ and $\eta$ follow.
\end{proof}

Now let $(F_i,r_i)$, $1\le i\le\ell$, satisfy the hypotheses of
\cref{lem:rooted-radial}, with positive valencies $d_i$, and put
\[
 G=F_1\square\cdots\square F_\ell,\qquad a=(r_1,\ldots,r_\ell).
\]
Let $b_i$ be the indicator of those neighbors of $a$ obtained by moving in
the $i$th factor only.  Thus $b=\sum_i b_i$, the $b_i$ are orthogonal, and
$\norm{b_i}^2=d_i$.

Write $\Theta_i,\omega_i,\Phi_i,\eta_i$ for the local data of $F_i$.
Let $I_1,\ldots,I_t$ be the equivalence classes of
\begin{equation}\label{eq:boundary-equivalence}
 i\sim j\quad\Longleftrightarrow\quad\eta_i(z)=\eta_j(z)
 \text{ for all real }z.
\end{equation}
We call these the \emph{boundary classes}.  For each class put
\begin{equation}\label{eq:class-data}
 c_g=\sum_{i\in I_g}b_i,\qquad
 T_g=\sum_{i\in I_g}\Theta_i,\qquad \Lambda=T_1+\cdots+T_t.
\end{equation}

\begin{theorem}[Mixed-product boundary classification]\label{thm:main}
In the notation above, the boundary-return space in the tensor product
of the rooted radial spaces, with the root removed, is exactly
\begin{equation}\label{eq:boundary-classification}
 \calC=\Span\{c_1,\ldots,c_t\}.
\end{equation}
For $\lambda\in\Lambda$, let $R_\lambda$ have
rows
\begin{equation}\label{eq:fibre-rows}
 (1,y_1,\ldots,y_t),\qquad
 y\in T_1\times\cdots\times T_t,\quad\sum_g y_g=\lambda.
\end{equation}
Then
\begin{equation}\label{eq:fibre-main}
 \dim P_a(G)=\sum_{\lambda\in\Lambda}\rank R_\lambda-1.
\end{equation}
In particular, this dimension is determined by the rooted local spectral
measures of the factors.
\end{theorem}

\begin{proof}
We separate the boundary classification from the final rank computation.

\smallskip\noindent\emph{The tensor product model.}
Let $\widetilde V=\bigotimes_i W(F_i,r_i)$.  It contains $e_a$ and the
all-ones vector and is invariant under the product adjacency matrix $A$.
It is also invariant under the root projection and the first-sphere
projection.  Consequently $V=\widetilde V\cap e_a^\perp$ is invariant
under $B$ and $E$ and contains the residual all-ones vector.  Its boundary
space is
\[
 \bdy=EV=\Span\{b_i:1\le i\le\ell\}.
\]
Tensoring the identifications in \cref{lem:rooted-radial} identifies
$\widetilde V$ with $L^2(\prod_i\Theta_i,\bigotimes_i\omega_i)$.  In this
model, $e_a$ is the constant function $1$, adjacency is multiplication by
\[
 \tau=X_1+\cdots+X_\ell,
\]
and $b_i$ is the coordinate function $X_i$.  The variables $X_i$ are
independent, centered, and have variances $d_i$.  The residual space is
$1^\perp$, and $B$ is the compression of multiplication by $\tau$ to that
space.  Let $\pi$ be the orthogonal projection onto $\bdy$.

\smallskip\noindent\emph{The boundary exponential kernel.}
By \cref{lem:full-returns}, $\calC$ is also the smallest boundary subspace
containing $b=\tau$ and invariant under all $\pi A^k\pi$.  Equivalently,
it is invariant under $\pi e^{zA}\pi$ for every real $z$.  Put
$\sigma_i=\sqrt{d_i}$ and use the orthonormal basis
$u_i=X_i/\sigma_i$ of $\bdy$.  Independence gives
\begin{equation}\label{eq:kernel}
 \frac{\pi e^{zA}\pi|_{\bdy}}{\prod_i\Phi_i(z)}
 =D(z)+w(z)w(z)^{\mathsf T},\qquad
 D(z)=\diag(\eta_1'(z),\ldots,\eta_\ell'(z)),\quad
 w_i(z)=\sigma_i\eta_i(z).
\end{equation}
Indeed, an off-diagonal entry is
$\Phi_i'/\Phi_i\cdot\Phi_j'/\Phi_j/(\sigma_i\sigma_j)$,
and the $i$th diagonal entry is
$\Phi_i''/(d_i\Phi_i)=\eta_i'+d_i\eta_i^2$.
The denominators in \eqref{eq:kernel} are strictly positive on $\R$.

The vector $w(z)$ belongs to $\calC$ for every real $z$.  To see this,
use $Ae_a=b$ and $\pi e_a=0$ to obtain
\[
 \pi e^{zA}e_a
 =\sum_{k\ge1}\frac{z^k}{k!}\pi A^{k-1}\pi b\in\calC.
\]
Its coordinates are $(\prod_i\Phi_i(z))w(z)$, proving the claim.

\smallskip\noindent\emph{Separation of the boundary classes.}
All operators in \eqref{eq:kernel} are symmetric.  Therefore
$\calC^\perp$, with orthogonal complement taken inside $\bdy$, is invariant
under them.  Since $w(z)\in\calC$, the rank-one term annihilates
$\calC^\perp$, which is consequently invariant under every $D(z)$.
The joint coordinate eigenspaces of the diagonal family $\{D(z):z\in\R\}$
are exactly the coordinate blocks $I_g$.  In fact,
$\eta_i'=\eta_j'$ identically implies $\eta_i=\eta_j$, by
$\eta_i(0)=\eta_j(0)=0$.

For completeness, these joint block projections belong to the algebra
generated by finitely many $D(z)$.  Choose, for every pair of different
classes, a real argument at which their derivatives differ.  The resulting
finite tuples of diagonal values are distinct for different classes.
A real linear combination of the chosen diagonal matrices can be chosen
with distinct values on all classes, since only finitely many proper
hyperplanes of coefficients must be avoided.  Univariate interpolation
then gives each class projection.  Hence $\calC^\perp$ is the orthogonal
direct sum of its intersections with the class coordinate blocks.

In the $u_i$ basis, $b$ is the vector $\sigma=(\sigma_i)_i$.  Since
$b\in\calC$, every class component of $\calC^\perp$ is orthogonal to the
corresponding vector $\sigma_g=(\sigma_i\mathbf1_{i\in I_g})_i$.
It follows that $\calC$ contains every $\sigma_g$.  Conversely,
$S=\Span\{\sigma_g:1\le g\le t\}$ contains $\sigma$, is invariant under
all $D(z)$, and contains all $w(z)$ because the functions $\eta_i$ are
constant within each class.  Equation \eqref{eq:kernel} shows that $S$ is
invariant under every full return, and thus under every residual return.
Minimality yields $\calC\subseteq S$.  Therefore $\calC=S$, which is
\eqref{eq:boundary-classification} in vertex coordinates.

\smallskip\noindent\emph{The spectral fibre formula.}
Adjoin the root to the principal module.  The block form
\eqref{eq:adjacency-block} and \cref{lem:transfer} give
\begin{equation}\label{eq:adjoined}
 \R e_a\oplus P_a
 =\R[A]\Span\{e_a,c_1,\ldots,c_t\}.
\end{equation}
Here is a verification of both inclusions.  The left side is $A$-invariant
and contains all the displayed seeds.  Conversely, the right side contains
$e_a$ and hence is invariant under its orthogonal complement projection;
its residual part is $B$-invariant and contains $\calC$.  It therefore
contains $P_a=\R[B]\calC$.

In the spectral model the seeds in \eqref{eq:adjoined} are
$1,Y_1,\ldots,Y_t$, where $Y_g=\sum_{i\in I_g}X_i$.  Distinct $Y_g$ are
independent, and each has positive weight at every point of $T_g$.
Spectral interpolation in $\tau=\sum_gY_g$ decomposes the right side of
\eqref{eq:adjoined} into its restrictions to the fibres $\tau=\lambda$.
On such a fibre its dimension is the rank of the evaluations of the seed
functions.  The positive spectral weights multiply evaluation rows by
nonzero factors and do not change rank.  Repeated decompositions of a
class sum likewise repeat rows and do not change rank.  Thus the rank is
exactly $\rank R_\lambda$.  Sum over the fibres and remove the
one-dimensional root summand in \eqref{eq:adjoined}.
\end{proof}

\begin{remark}[Finite determination of the classes]\label{rem:finite-classes}
The analytic notation is not an additional infinite-data assumption.
Integrating \eqref{eq:boundary-equivalence}, using $\Phi_i(0)=1$, gives
\begin{equation}\label{eq:Phi-power}
 i\sim j\quad\Longleftrightarrow\quad
 \Phi_i(z)^{d_j}=\Phi_j(z)^{d_i}\quad\text{for all real }z.
\end{equation}
Both sides are finite exponential sums.  Collecting equal exponents reduces
this to finitely many exact coefficient comparisons, because exponential
functions with distinct real exponents are linearly independent.  This
also explains why equal supports without weights are insufficient for
classifying general mixed products.  No polynomial running-time bound
for this expansion is asserted.
\end{remark}

\begin{corollary}\label{cor:fibre-bounds}
Under the hypotheses of \cref{thm:main},
\[
 |\Lambda|-1\le\dim P_a\le
 \sum_{\lambda\in\Lambda}
 \min\{t,\,|\{y\in\textstyle\prod_gT_g:\sum_g y_g=\lambda\}|\}-1.
\]
For a single boundary class, $P_a=\R[B]\one$ and
$\dim P_a=|\Lambda|-1$.
\end{corollary}
\begin{proof}
Every fibre is nonempty and its row matrix has rank at least one.  Its
rows satisfy the linear relation
$-\lambda\cdot1+y_1+\cdots+y_t=0$, so the rank is at most $t$ and at most
the number of rows.  With one class, \cref{thm:main} gives
$\calC=\R b$, and \cref{lem:transfer,prop:deletion} give the cyclic-space
assertion.
\end{proof}

\section{Additive consequences and explicit mixed products}\label{sec:additive}

\subsection{Powers and the diameter bound}
For a finite subset $S\subseteq\R$, write
$mS=\{s_1+\cdots+s_m:s_i\in S\}$, with repetitions allowed.

\begin{lemma}[The equality case for sumset growth]\label{lem:sumset}
If $S\subseteq\R$ has $q\ge2$ elements, then
$|mS|\ge m(q-1)+1$.  For $m\ge2$, equality holds if and only if $S$ is an
arithmetic progression.
\end{lemma}
\begin{proof}
For $X=\{x_0<\cdots<x_r\}$ and $Y=\{y_0<\cdots<y_s\}$, the chain
\[
 x_0+y_0<x_1+y_0<\cdots<x_r+y_0<x_r+y_1<\cdots<x_r+y_s
\]
proves $|X+Y|\ge|X|+|Y|-1$.  Iteration gives the bound.  Equality for
$m\ge2$ forces $|2S|=2|S|-1$.

Translate $S$ to write $S=\{0=s_0<s_1<\cdots<s_r\}$, $r=q-1$.  Under
this equality, the $2r+1$ elements
\[
 0,s_1,\ldots,s_r,s_r+s_1,\ldots,2s_r
\]
exhaust $2S$.  For $0\le i<r$, the element $s_i+s_1$ is smaller than
$s_r+s_1$.  The displayed list has no element strictly between $s_r$ and
$s_r+s_1$, so $s_i+s_1\in S$.  These $r$ positive elements of $S$ are
strictly increasing, and must therefore be $s_1,\ldots,s_r$ in order.
Hence $s_{i+1}=s_i+s_1$.  Conversely an arithmetic progression of length
$q$ has an $m$-fold sumset of length $m(q-1)+1$.
\end{proof}

\begin{theorem}[Powers]\label{thm:powers}
Let $(F,r)$ be a connected regular rooted graph with an equitable distance
partition, let $D=\ecc(r)\ge1$, and let $\Theta$ be its local spectral
support.  For $a=(r,\ldots,r)$ in $F^{\square m}$,
\begin{equation}\label{eq:powers}
 P_a=\R[B]\one,\qquad \dim P_a=|m\Theta|-1\ge mD.
\end{equation}
For $m\ge2$, equality with $mD$ holds if and only if $\Theta$ is an
arithmetic progression.  For a distance-regular graph, $D=\diam(F)$ and
$\Theta$ is the full distinct adjacency spectrum.
\end{theorem}
\begin{proof}
Every factor has the same rooted local measure, so there is a single
boundary class.  Apply \cref{cor:fibre-bounds} and then
\cref{lem:rooted-radial,lem:sumset}.  For a distance-regular graph, write $A_i$ for its distance-$i$ matrix.
The recurrence
$AA_i=b_{i-1}A_{i-1}+a_iA_i+c_{i+1}A_{i+1}$ expresses $A_i$ as a
polynomial of degree $i$ in $A$.  At $i=D$ it gives
$(A-a_DI)A_D-b_{D-1}A_{D-1}=0$, a polynomial relation of degree $D+1$.
The root cyclic module already has dimension $D+1$, so the local
support is the full distinct spectrum.
\end{proof}

\begin{example}[The minimum dimension need not imply distance-regularity]
For $r\ge3$, the graph $F=K_{r,r}$ is distance-regular of diameter two and
has spectrum $\{r,0,-r\}$.  Therefore
$\dim P_a(F^{\square m})=2m$.  For $m\ge2$ the product is not
distance-regular: at distance two from the root, a vertex whose factor
distances are $(2,0,\ldots,0)$ has $r$ neighbors in the preceding sphere,
whereas one of type $(1,1,0,\ldots,0)$ has two.  Thus minimum growth of the
principal module is weaker than distance-regularity of the product.
\end{example}

\begin{example}[A non-arithmetic spectrum]
The five-cycle has distinct spectrum
$\{2,(\sqrt5-1)/2,-(\sqrt5+1)/2\}$.  Its six unordered two-term sums are
distinct.  Hence $\dim P_a(C_5\square C_5)=5$, although the product has
diameter four.
\end{example}

\subsection{Two complete-graph classes}
The fibre formula has a particularly explicit interpretation when there
are two boundary classes.  Each fibre lies on a line, so its rank is one
for a single point and two for two or more distinct points.  Consequently
the first loss occurs at a fibre containing three points.

\begin{theorem}[Mixed complete-graph powers]\label{thm:complete}
Let $p,q\ge2$ be distinct integers, let $m,n\ge1$, and put
\[
 G=K_p^{\square m}\square K_q^{\square n},\qquad
 g=\gcd(p,q),\quad p'=p/g,\quad q'=q/g.
\]
Then, at every root,
\begin{align}
 \dim P_a(G)
 &= (m+1)(n+1)-(m-2q'+1)_+(n-2p'+1)_+-1,\label{eq:complete-dim}\\
 \defect(G,a)
 &= (m-2q'+1)_+(n-2p'+1)_+.\label{eq:complete-defect}
\end{align}
In particular $P_a=U_a^-$ if and only if $m<2q'$ or $n<2p'$.
\end{theorem}
\begin{proof}
At any root of $K_q$, the local measure has weights $1/q$ and $(q-1)/q$
at $q-1$ and $-1$, respectively.  Direct differentiation gives
\[
 \Phi_q(z)=\frac{e^{(q-1)z}+(q-1)e^{-z}}{q},\qquad
 \eta_q(z)=\frac{e^{qz}-1}{e^{qz}+q-1},\qquad
 \eta_q''(0)=q-2.
\]
Thus the $K_p$ copies form one boundary class and the $K_q$ copies form
the other.  The class support pairs are indexed by the rectangle
\[
 \mathcal R=\{(i,j):0\le i\le m,\ 0\le j\le n\},
\]
where
$y_1=m(p-1)-pi$ and $y_2=n(q-1)-qj$.
Fixing their sum is equivalent to fixing $pi+qj$.
The integral points in a nonempty fibre are consecutive points on an
arithmetic path with step $(q',-p')$.  To justify consecutiveness, any two
integral solutions differ by an integral multiple of this primitive step,
and all intermediate points remain inside the rectangle by convexity.

If a fibre has $r$ points, its matrix in \eqref{eq:fibre-rows} has rank
$\min\{2,r\}$: two different points give independent rows because the
first coordinate is $1$.  The loss from the total number of rectangle
points is therefore $\sum(r-2)_+$ over the fibres.  A path of $r$ points
contains exactly $(r-2)_+$ triples of consecutive points.  Such triples
are uniquely specified by a starting point $(i,j)$ satisfying
\[
 0\le i\le m-2q',\qquad 2p'\le j\le n.
\]
Their total number is $(m-2q'+1)_+(n-2p'+1)_+$.  Substitution in
\cref{thm:main} proves \eqref{eq:complete-dim}.

A complete graph with at least two vertices is Cartesian-prime: a product
with two nontrivial factors contains pairs differing in two coordinates
that are not adjacent.  By the Cartesian automorphism theorem
\cite{HIK2011}, automorphisms of $G$ only permute the $K_p$ coordinates
among themselves and the $K_q$ coordinates among themselves, besides
acting within coordinates.  At a fixed root, an orbit is determined by
the number of non-root entries in each of these two blocks.  Thus
$\dim U_a^-=(m+1)(n+1)-1$.  Subtract \eqref{eq:complete-dim}; the last
assertion uses $P_a\subseteq U_a^-$ from \cref{lem:automorphisms}.
\end{proof}

\begin{example}
For $K_2^{\square3}\square K_3^{\square2}$, the adjacency cyclic space has
dimension $10$, whereas $\dim P_a=\dim U_a^-=11$.  There are twelve
rectangle points and exactly one two-point fibre, so the second boundary
class resolves the collision.  For
$K_2^{\square6}\square K_3^{\square4}$, the rectangle has $35$ points;
there is exactly one triple of consecutive points on its spectral
fibres.  Here $\dim P_a=33$ and $\dim U_a^-=34$.  The obstruction is
therefore not the existence of collisions, but insufficient affine rank
within a collision fibre.
\end{example}

\section{Distance-regular graphs and scalar moment reconstruction}\label{sec:dr}

Let $G$ be distance-regular of valency $d$ and diameter $D\ge1$.  We use
the intersection numbers $b_i,c_i,a_i$, where a vertex at distance $i$
from a root has respectively $b_i,c_i,a_i$ neighbors at distances
$i+1,i-1,i$.  Thus
\[
 c_0=0,\quad b_0=d,\quad c_1=1,\quad b_D=0,\quad
 a_i=d-b_i-c_i.
\]
For a fixed root put $f_i=\one_{S_i(a)}$ and $\kappa_i=|S_i(a)|$.
Counting edges between consecutive spheres gives
$\kappa_i b_i=\kappa_{i+1}c_{i+1}$.

\begin{proposition}[The primary punctured module]\label{prop:dr-module}
For every root of a distance-regular graph,
\begin{equation}\label{eq:dr-module}
 P_a=\R[B]\one=\Span\{f_1,\ldots,f_D\},\qquad
 \Balg_a\big|_{P_a}=\End(P_a)\cong M_D(\R).
\end{equation}
On this module,
\begin{equation}\label{eq:dr-action}
 Bf_i=b_{i-1}f_{i-1}+a_if_i+c_{i+1}f_{i+1},\qquad
 E=\frac{bb^{\mathsf T}}d,
\end{equation}
with the terms outside $1\le i\le D$ omitted.
\end{proposition}
\begin{proof}
The first equality and the dimension follow from
\cref{lem:rooted-radial,prop:deletion}; equitability also gives the
$B,E$-invariance of the residual sphere span.  The first sphere indicator
is $b$, of squared norm $d$, which proves the formula for $E$.  The
adjacency action is the usual intersection-number recurrence.

In the orthonormal basis $f_i/\sqrt{\kappa_i}$, $B$ is symmetric
tridiagonal with positive off-diagonal entries and $E$ is the projection
onto the first coordinate.  Thus $b,Bb,\ldots,B^{D-1}b$ are a basis.
The operators
\[
 B^rEB^s=\frac{1}{d}(B^rb)(B^sb)^{\mathsf T},\qquad0\le r,s<D,
\]
span all endomorphisms of $P_a$.  This proves the full-image assertion.
The use of an orthonormal basis here is important: the unnormalized
sphere basis has Gram matrix $\diag(\kappa_1,\ldots,\kappa_D)$.
\end{proof}

\subsection{Reconstructing the intersection array}
Define the scalar adjacency moments at a puncture by
\begin{equation}\label{eq:scalar-moments}
 s_j(G,a)=\one^{\mathsf T}B^j\one\qquad(j\ge0).
\end{equation}
For comparison, a principal adjacency--degree moment is
$\one^{\mathsf T}w(B,D(G-a))\one$, where $w$ is a word in two
noncommuting variables, including the empty word.

\begin{theorem}[Adjacency moments suffice]\label{thm:moments}
Let $G,G'$ be connected distance-regular graphs of positive diameter,
with arbitrary roots.  The following are equivalent: they have the same
intersection array; all moments in \eqref{eq:scalar-moments} are equal;
and all principal adjacency--degree moments are equal.
Moreover, among such graphs of diameter at most $L$, equality of
\begin{equation}\label{eq:finite-moments}
 s_j\qquad(0\le j\le2L-1)
\end{equation}
already implies equality of the intersection arrays.
\end{theorem}
\begin{proof}
An intersection array determines all sphere sizes, by
$\kappa_0=1$ and $\kappa_{i+1}=\kappa_i b_i/c_{i+1}$, and determines the
operators in \eqref{eq:dr-action}.  It therefore determines every stated
moment.  Equality of all adjacency--degree moments includes equality of
all adjacency moments.  It remains to establish the finite assertion.

Let $u=\one$ on the deletion and form the $L\times L$ matrices
\[
 H=(s_{i+j})_{0\le i,j<L},\qquad
 K=(s_{i+j+1})_{0\le i,j<L}.
\]
The first is the Gram matrix of $u,Bu,\ldots,B^{L-1}u$; the second gives
all pairings with $B$ inserted.  By \cref{prop:dr-module}, these vectors
span $P_a$, and $\rank H=D$.  Equality of the moments for the two graphs
makes
\[
 U\Big(\sum_{i=0}^{L-1}\alpha_iB^iu\Big)
   =\sum_{i=0}^{L-1}\alpha_i(B')^iu'
\]
a well-defined surjective isometry.  Equality of $K$ implies
$UB|_{P_a}=B'|_{P_{a'}}U$, since the pairings agree on spanning sets.
Also $Uu=u'$.  This argument includes the case $D<L$, because the null
space of the common Gram matrix accounts for all relations.

The number of vertices is $v=s_0+1$.  If $v=2$, the graph is $K_2$ and the
array is fixed.  If $v\ge3$, deletion removes $d$ edges from a
$d$-regular graph, so
\begin{equation}\label{eq:degree-reconstruct}
 d=\frac{s_1}{v-2}.
\end{equation}
Consequently $b=(dI-B)u$ is determined by the pointed operator $(P_a,B,u)$.
We now reconstruct all intersection numbers and sphere indicators
intrinsically in this pointed space.  Start with
\[
 f_1=b,\qquad\kappa_1=d,\qquad c_1=1.
\]
Given $f_i$ and the preceding data, set
\begin{equation}\label{eq:recover-ai}
 a_i=\frac{\ip{f_i}{Bf_i}}{\kappa_i},\qquad
 b_i=d-a_i-c_i.
\end{equation}
For $i<D$, subtract the known terms in the recurrence:
\[
 v_i=Bf_i-a_if_i-b_{i-1}f_{i-1}=c_{i+1}f_{i+1},
\]
where the last subtraction is omitted for $i=1$.  Since
$c_{i+1},\kappa_{i+1}>0$,
\begin{equation}\label{eq:recover-ci}
 c_{i+1}=\frac{\norm{v_i}^{2}}{\ip{u}{v_i}},\qquad
 \kappa_{i+1}=\frac{\ip{u}{v_i}^{2}}{\norm{v_i}^{2}},\qquad
 f_{i+1}=v_i/c_{i+1}.
\end{equation}
At $i=D$, the corresponding residual vector is zero.  All quantities in
\eqref{eq:degree-reconstruct}--\eqref{eq:recover-ci} are preserved by $U$.
Thus the common finite moments determine the entire array.
\end{proof}

\begin{remark}
The theorem concerns moments, not merely the spectrum of $B$.  The
weights of the seed $\one$ in the spectral decomposition matter.  It
also does not distinguish nonisomorphic distance-regular graphs having
the same intersection array.  In particular, principal moments cannot
separate a Doob graph from the Hamming graph with the same array.
\end{remark}

\section{Compressed Terwilliger corners and Hamming graphs}\label{sec:corners}

For any connected rooted graph, let $E_i^*$ be the diagonal projection
onto vertices at distance $i$ from the root, and let
\[
 \Talg(a)=\langle A,E_0^*,\ldots,E_D^*\rangle,
 \qquad D=\ecc(a).
\]
This is the Terwilliger, or subconstituent, algebra.  Put $e=I-E_0^*$ and
identify the corner $e\Talg(a)e$ with operators on the residual vertex
space.  In particular $E_1^*|_{e\R^{V(G)}}=E$.

\begin{lemma}[What a puncture omits]\label{lem:corner}
For a connected regular graph,
\begin{equation}\label{eq:corner-generators}
 e\Talg(a)e
   =\langle B,E_1^*|_e,\ldots,E_D^*|_e,bb^{\mathsf T}\rangle.
\end{equation}
If $G-a$ is connected, then $bb^{\mathsf T}\in\Balg_a$ and hence
\begin{equation}\label{eq:corner-distance}
 e\Talg(a)e=\langle\Balg_a,E_2^*|_e,\ldots,E_D^*|_e\rangle.
\end{equation}
In this case, equality with $\Balg_a$ holds if and only if all residual
distance projections belong to $\Balg_a$.
\end{lemma}
\begin{proof}
Insert $I=e+E_0^*$ between all factors in a compressed word in the
generators of $\Talg(a)$.  A term that stays in the residual space is a
word in $B$ and the residual distance projections.  A term visiting the
root uses only its one-dimensional space; entering and leaving it gives
\[
 eAE_0^*Ae=bb^{\mathsf T}.
\]
Since $E_0^*AE_0^*=0$, any nonzero term is a product of residual words and
such rank-one passages.  This proves inclusion in the right side of
\eqref{eq:corner-generators}; the reverse inclusion follows from the
displayed expressions for its generators.

If $G-a$ is connected, its Laplacian
$L_-=D(G-a)-B$ belongs to $\Balg_a$ and has kernel $\R\one$.
Spectral interpolation gives its orthogonal kernel projection
$\one\one^{\mathsf T}/(|V(G)|-1)$ as a polynomial in $L_-$.  Thus
$J_-=\one\one^{\mathsf T}\in\Balg_a$, and
$bb^{\mathsf T}=EJ_-E\in\Balg_a$.  Equation \eqref{eq:corner-distance}
and the criterion follow.
\end{proof}

\begin{corollary}\label{cor:srg-corner}
A distance-regular graph of positive diameter has a connected vertex
deletion at every root.  Its punctured algebra equals the compressed
Terwilliger algebra if its diameter is at most two.  In particular this
holds for every connected nontrivial strongly regular graph.
\end{corollary}
\begin{proof}
Let $D$ be the diameter.  If deletion of $a$ disconnects the graph, choose
a vertex $v$ at distance $D$ from $a$, which exists because the graph is
distance-regular, and a neighbor $u$ of $a$ in a different residual
component.  Every $u$--$v$ path passes through $a$, giving distance $D+1$,
a contradiction.  The one-vertex deletion of $K_2$ is connected as well.
For $D=1$ the only residual distance projection is $e$; for $D=2$ they are
$E$ and $e-E$.  Apply \cref{lem:corner}.
\end{proof}

The corner equality in diameter two thus has an elementary explanation:
root passages are already available from the residual Laplacian, and
there is only one additional distance class.  It is not a consequence
of a classification of all irreducible strongly regular Terwilliger
modules.  In higher diameter the missing distance projections can create
a genuine loss, as the next theorem determines exactly for hypercubes.

\subsection{The Boolean lattice decomposition}
We recall the decomposition needed for the hypercube and give its proof
to specify the residual action.  The underlying full Terwilliger
structure is classical; see Go~\cite{Go2002}.
Identify $V(Q_d)$ with the subsets of $[d]$, with root $\varnothing$.
Let $V_k$ be the span of the $k$-subsets.  Define raising and lowering
operators by
\[
 R e_S=\sum_{i\notin S}e_{S\cup\{i\}},\qquad L=R^{\mathsf T}.
\]
Thus $A=L+R$.  Direct counting gives
\begin{equation}\label{eq:boolean-commutator}
 (LR-RL)|_{V_k}=(d-2k)I.
\end{equation}
Indeed, the off-diagonal terms on either side correspond to replacing
one element of $S$ by one outside $S$ and cancel, while the diagonal
terms count $d-k$ and $k$ choices.

\begin{lemma}[Harmonic chains]\label{lem:boolean}
For $0\le j\le\lfloor d/2\rfloor$, let $H_j=\ker L\cap V_j$ and put
$h_j=\binom dj-\binom d{j-1}$, with $\binom d{-1}=0$.
Then $\dim H_j=h_j$.  An orthogonal basis of each $H_j$ generates
mutually orthogonal chains
\[
 W(h)=\Span\{h,Rh,\ldots,R^{d-2j}h\},\qquad h\in H_j,
\]
whose direct sum is the whole vertex space.  In an orthonormal chain
basis indexed by $0\le i\le d-2j$, adjacency is tridiagonal with zero
diagonal and consecutive off-diagonal entries
\begin{equation}\label{eq:boolean-jacobi}
 \sqrt{(i+1)(d-2j-i)}.
\end{equation}
The distance projection $E_k^*$ restricts to the coordinate projection
at $i=k-j$ when that index is in the chain, and to zero otherwise.
The adjacency spectrum of this chain is
\begin{equation}\label{eq:chain-spectrum}
 d-2j,\ d-2j-2,\ldots,\ 2j-d.
\end{equation}
\end{lemma}
\begin{proof}
For $v\in V_{j-1}$, \eqref{eq:boolean-commutator} gives
\[
 \norm{Rv}^2=\norm{Lv}^2+(d-2j+2)\norm v^2.
\]
For $1\le j\le\lfloor d/2\rfloor$, this makes
$R:V_{j-1}\to V_j$ injective and its transpose $L:V_j\to V_{j-1}$
surjective.  Hence $\dim H_j=h_j$, including the evident case $j=0$.

For $h\in H_j$, induction using \eqref{eq:boolean-commutator} proves
\begin{equation}\label{eq:chain-lowering}
 LR^ih=i(d-2j-i+1)R^{i-1}h.
\end{equation}
Taking inner products gives
\[
 \norm{R^ih}^2=i(d-2j-i+1)\norm{R^{i-1}h}^2.
\]
Thus the chain vectors are nonzero through $i=d-2j$ and
$R^{d-2j+1}h=0$.  Normalization gives \eqref{eq:boolean-jacobi}.
The assertion for distance projections follows from
$R^ih\in V_{j+i}$.

Vectors in different levels are orthogonal.  At the same level $k$, if
$h\in H_j$ and $h'\in H_{j'}$ with $j<j'$, repeated lowering shows that
\[
 \ip{R^{k-j}h}{R^{k-j'}h'}
 \quad\text{is a scalar multiple of}\quad
 \ip{R^{j'-j}h}{h'}=\ip{h}{L^{j'-j}h'}=0.
\]
For $j=j'$, the same operation gives a positive multiple of
$\ip{h}{h'}$.  Hence the chains are mutually orthogonal.  At level $k$
their total dimension is
\[
 \sum_{j=0}^{\min\{k,d-k\}}h_j
 =\binom d{\min\{k,d-k\}}=\binom dk,
\]
so they fill the vertex space.

Finally put $s=d-2j$.  The matrix in \eqref{eq:boolean-jacobi} is the
restriction of $\sum_{r=1}^s X_r$ to the symmetric tensors in
$(\R^2)^{\otimes s}$, where
$X=\left(\begin{smallmatrix}0&1\\1&0\end{smallmatrix}\right)$ acts in
coordinate $r$.  In the normalized symmetric basis with $i$ entries of
the second standard basis vector, its consecutive coefficients are
$\sqrt{(i+1)(s-i)}$.  Taking instead the eigenbasis of $X$ gives one
symmetric eigenvector with eigenvalue $s-2i$ for each $0\le i\le s$.
This proves \eqref{eq:chain-spectrum}, including $s=0$.
\end{proof}

\begin{theorem}[The full punctured hypercube algebra]\label{thm:cube}
For $d\ge2$ and any root of $Q_d$,
\begin{align}
 \Balg_a(Q_d)&\cong
 M_d(\R)\oplus M_{d-1}(\R)\oplus\R^{\max\{d-3,0\}},\label{eq:cube-algebra}\\
 e\Talg(a)e&\cong
 M_d(\R)\oplus
 \bigoplus_{j=1}^{\lfloor d/2\rfloor}M_{d-2j+1}(\R).
 \label{eq:cube-corner}
\end{align}
For $d=1$ both algebras are $\R$.  Therefore the two algebras are equal
as subalgebras of the residual matrix algebra precisely for $1\le d\le4$.
For $d\ge2$,
\[
 \dim\Balg_a=d^2+(d-1)^2+\max\{d-3,0\}.
\]
\end{theorem}
\begin{proof}
By vertex-transitivity we may take root $\varnothing$.  The chain with
$j=0$ is the primary chain and is the only chain meeting the root level.
After deletion it has dimension $d$.  Its residual adjacency matrix is
irreducible tridiagonal, and $E=E_1^*$ is the rank-one projection at its
first remaining level.  As in \cref{prop:dr-module}, these two operators
generate $M_d(\R)$.

Each chain with $j=1$ survives intact, has dimension $d-1$, and has $E$
as the rank-one projection onto its first level.  It therefore gives a
full $M_{d-1}(\R)$ image.  The $h_1=d-1$ copies have identical generator
matrices and contribute one simple algebra component, not $h_1$ different
components.

On every chain with $j\ge2$, the generator $E$ vanishes.  The simultaneous
image on the sum of all these chains is therefore the polynomial algebra
of their direct-sum adjacency operator.  By \eqref{eq:chain-spectrum},
these spectra are nested and their union, when $d\ge4$, is
\[
 \{d-4,d-6,\ldots,4-d\},
\]
of cardinality $d-3$.  Spectral interpolation identifies this image with
$\R^{d-3}$.  There are no such chains when $d=2,3$.

We justify that these images are independent algebra components, rather
than merely separately surjective restrictions.  A finite-dimensional
transpose-closed real matrix algebra is semisimple.  One elementary
reason is that its Jacobson radical is transpose-stable and nilpotent;
if $x$ lies in the radical, then $x^{\mathsf T}x$ is both positive
semidefinite and nilpotent, forcing $x=0$.  The irreducible constituents
just exhibited have full real matrix images.  The two visible
constituents have different dimensions, and the remaining ones are
one-dimensional with $E=0$ and distinct adjacency eigenvalues.  They are
therefore pairwise inequivalent.  Their kernels are distinct maximal
two-sided ideals, so the Chinese remainder theorem makes the map onto
the product of these images surjective.  Its kernel is zero because the
chain decomposition is exhaustive.  This proves \eqref{eq:cube-algebra}.

For the full Terwilliger algebra, each chain in \cref{lem:boolean} has
all its level projections available and gives its full matrix algebra.
Different $j$ have different first nonzero levels, so their representations
are inequivalent.  The same semisimplicity argument gives the classical
direct-sum decomposition, with one component for each $j$; the harmonic
multiplicities do not alter the algebra dimension.  Compressing by $e$
removes the first coordinate only from the $j=0$ component, proving
\eqref{eq:cube-corner}.

There is always an inclusion $\Balg_a\subseteq e\Talg(a)e$.  For $d=2,3$
there are no hidden chains, and for $d=4$ there is just a single hidden
one-dimensional chain.  Hence the displayed dimensions agree for
$1\le d\le4$, proving equality.  For $d\ge5$, the $j=2$ chain has
dimension $d-3\ge2$.  The corner has its full noncommutative matrix image
on that chain, whereas $\Balg_a$ has only a commutative polynomial image.
The inclusion is therefore strict.
\end{proof}

\begin{remark}
The mechanism is an endpoint obstruction.  The first-sphere projection
sees only the chains starting in levels zero and one.  It vanishes on
all later chains, which are then distinguished only by adjacency
eigenvalues.  In particular $\dim P_a(Q_d)=d$ for every $d$, but this does
not imply that $\Balg_a$ has only one simple component or that it equals
the compressed corner.
\end{remark}

\subsection{All Hamming graphs}
The same endpoint analysis extends to $H(d,q)=K_q^{\square d}$.
For $q\ge3$ an additional local constituent is present, and the threshold
for corner equality changes.  The full Terwilliger decomposition used
below is the one studied by Levstein, Maldonado, and
Penazzi~\cite{LevsteinMaldonadoPenazzi2006}; we derive the needed chain
matrices directly from the one-coordinate decomposition.

\begin{theorem}[The full punctured Hamming algebra]\label{thm:hamming}
Let $q\ge3$.  For $d\ge2$,
\begin{equation}\label{eq:hamming-algebra}
 \Balg_a(H(d,q))\cong
 M_d(\R)\oplus M_d(\R)\oplus M_{d-1}(\R)\oplus\R^{d-1}.
\end{equation}
For $d=1$ it is $\R\oplus\R$.  In particular,
\begin{equation}\label{eq:hamming-dimension}
 \dim\Balg_a(H(d,q))=3d^2-d\qquad(d\ge1).
\end{equation}
For all $d\ge1$, the compressed Terwilliger algebra is
\begin{equation}\label{eq:hamming-corner}
 e\Talg(a)e\cong M_d(\R)\oplus
 \bigoplus_{\substack{0\le r\le d,\ 0\le j\le\lfloor(d-r)/2\rfloor\\
                     (r,j)\ne(0,0)}}
 M_{d-r-2j+1}(\R).
\end{equation}
Consequently, among all Hamming graphs with $d\ge1,q\ge2$,
\begin{equation}\label{eq:hamming-equality}
 \Balg_a(H(d,q))=e\Talg(a)e
 \quad\Longleftrightarrow\quad
 \begin{cases}
 d\le4,&q=2,\\
 d\le2,&q\ge3.
 \end{cases}
\end{equation}
\end{theorem}
\begin{proof}
Fix root symbol $0$ in each coordinate.  In a single copy of $\R^q$, let
\[
 u_0=e_0,\qquad
 u_1=(q-1)^{-1/2}\sum_{x\ne0}e_x,\qquad
 W=\Span\{u_0,u_1\},\qquad Z=W^\perp.
\]
The adjacency of $K_q$ and the projection onto non-root symbols preserve
$W\oplus Z$.  On $W$ their matrices are respectively
\[
 C_q=\begin{pmatrix}0&\sqrt{q-1}\\\sqrt{q-1}&q-2\end{pmatrix},
 \qquad\begin{pmatrix}0&0\\0&1\end{pmatrix},
\]
and on $Z$ they act as $-I$ and $I$.  Here $\dim Z=q-2>0$.

Choose the $r$ coordinates that lie in $Z$, and choose basis vectors
in their $Z$ spaces.  The remaining $s=d-r$ coordinates form
$W^{\otimes s}$, naturally identified with the Boolean tensor space of
\cref{lem:boolean}.  On this space the full Hamming adjacency is
\begin{equation}\label{eq:hamming-tensor-action}
 \sqrt{q-1}(R+L)+(q-2)K-rI,
\end{equation}
where $K$ is multiplication by the Boolean level.  The total distance
from the root is $r+K$.  Apply \cref{lem:boolean}, with $s$ in place of
$d$, to decompose this space into chains indexed by
$0\le j\le\lfloor s/2\rfloor$.  A chain has basis positions
$0\le i\le\ell$, where $\ell=d-r-2j$, and its matrices are
\begin{align}
 A_{ii}&=(q-2)(j+i)-r,\label{eq:hamming-chain-diagonal}\\
 A_{i,i+1}&=\sqrt{(q-1)(i+1)(\ell-i)},\label{eq:hamming-chain-offdiag}\\
 E_k^*|_{\mathrm{chain}}&=
 \text{the projection at }i=k-r-j\text{ when }0\le i\le\ell,
 \label{eq:hamming-chain-level}
\end{align}
with zero otherwise.  Every pair $(r,j)$ in the stated range occurs:
its multiplicity in the standard module is
\[
 \binom dr(q-2)^r
 \left(\binom{d-r}{j}-\binom{d-r}{j-1}\right)>0.
\]
The construction is an orthogonal decomposition of the entire tensor
space, so no other constituents are missing.

The adjacency matrix on a chain is a scalar shift of the action of
$\sum_{k=1}^{\ell}(C_q)_k$ on symmetric tensors of degree $\ell$.
Since the eigenvalues of $C_q$ are $q-1$ and $-1$, the chain spectrum is
\begin{equation}\label{eq:hamming-chain-spectrum}
 \{d(q-1)-q(r+j+i):0\le i\le d-r-2j\}.
\end{equation}
Indeed the scalar shift is $(q-2)j-r$, and adding it to
$\ell(q-1)-qi$ gives the displayed formula.

All distance projections are available to the full Terwilliger algebra,
so each chain gives a full matrix image.  The first and last nonzero
levels of a chain are $r+j$ and $d-j$.  These determine $j$ and $r$, so
distinct pairs have inequivalent Terwilliger representations.  The
semisimplicity and independent-component argument in \cref{thm:cube}
therefore gives one full matrix component for each pair.  Only $(0,0)$
meets the root, and its compression has size $d$, proving
\eqref{eq:hamming-corner}.

We now retain only residual adjacency and $E=E_1^*$.  The primary chain
$(0,0)$, with its root removed, gives $M_d(\R)$.  The only other chains
on which $E$ is nonzero are $(1,0)$, of dimension $d$, and $(0,1)$, of
dimension $d-1$; the latter exists only for $d\ge2$.  On each, $E$ is the
first-coordinate projection, so the positive off-diagonal entries in
\eqref{eq:hamming-chain-offdiag} give the full matrix image.
The two size-$d$ representations are inequivalent: on the primary
punctured chain $EBE=(q-2)E$, whereas on $(1,0)$ it is $EBE=-E$.
The $(0,1)$ representation has different dimension, and hence is
inequivalent to both.  This distinction between the two size-$d$ blocks
is necessary; separate surjectivity on equal-sized blocks alone would
not prove their independence.

On all remaining chains $r+j\ge2$, so $E=0$ and the image is only the
polynomial algebra of their simultaneous adjacency operator.  For
$d\ge2$, their union of spectra is exactly
\[
 \{d(q-1)-qk:2\le k\le d\}.
\]
Containment follows from \eqref{eq:hamming-chain-spectrum}; the reverse
containment already follows from the single chain $(r,j)=(2,0)$.
There are $d-1$ distinct values, giving $\R^{d-1}$ by spectral
interpolation.  These scalar constituents have $E=0$ and distinct
adjacency eigenvalues, and so are inequivalent to one another and to
the visible constituents.  The same independent-component argument
proves \eqref{eq:hamming-algebra}.  At $d=1$, only the first two
one-dimensional visible constituents are present, giving $\R^2$.
The dimension formula follows.

For $d=1,2$, the only hidden chains are absent or one-dimensional;
\eqref{eq:hamming-algebra} and \eqref{eq:hamming-corner} therefore have
the same dimension, and the algebras are equal by inclusion.  For
$d\ge3$, the hidden chain $(2,0)$ has dimension $d-1\ge2$.  The corner
acts on it as a full matrix algebra while the punctured algebra acts
commutatively, proving strict inclusion.  Combine this with
\cref{thm:cube} for $q=2$.
\end{proof}

\section{Stabilizer orbits and distance-graded defects}\label{sec:orbits}

The mixed-product theorem has no primeness assumption.  Orbit counting is
different: extra automorphisms can interchange Cartesian prime coordinates
hidden inside a nonprime factor.  In this section, whenever an exact
product orbit formula is used, the factors are explicitly required to be
Cartesian-prime.

\subsection{The Cartesian orbit formula}
For a rooted connected graph $(F,r)$, write $\mathcal O(F,r)$ for the set
of $\Aut(F)_r$-orbits on all vertices, including the root orbit, and put
$t(F,r)=|\mathcal O(F,r)|$.  Every such orbit lies in one distance sphere;
write $\rho(O)$ for that distance.

\begin{proposition}[Product orbit types]\label{prop:orbit-product}
Let $F_1,\ldots,F_s$ be pairwise nonisomorphic connected nontrivial
vertex-transitive Cartesian-prime graphs.  Let $m_i\ge1$ and
$G=\mathop{\square}_{i=1}^s F_i^{\square m_i}$.  Choose roots $r_i$ and
let $a$ be the tuple with root $r_i$ in every copy of $F_i$.
Set $t_i=t(F_i,r_i)$.  Then
\begin{equation}\label{eq:orbit-dimension}
 \dim U_a^-+1=\prod_{i=1}^s\binom{t_i+m_i-1}{m_i}.
\end{equation}
More precisely, the distance-graded orbit profile
\[
 \Omega_{G,a}(z)=\sum_{k\ge0}
 |\Aut(G)_a\backslash S_k(a)|\,z^k
\]
satisfies
\begin{equation}\label{eq:orbit-profile}
 \Omega_{G,a}(z)=
 \prod_{i=1}^s
 \left(
 \sum_{\substack{\alpha\in\mathbb Z_{\ge0}^{\mathcal O(F_i,r_i)}\\
                  \sum_O\alpha_O=m_i}}
 z^{\sum_O\alpha_O\rho(O)}
 \right).
\end{equation}
The dimension and profile do not depend on the root of $G$.
\end{proposition}
\begin{proof}
The automorphism theorem for connected Cartesian products
\cite{HIK2011} says that automorphisms act by coordinate automorphisms
and by permutations of isomorphic Cartesian prime factors.  For the
chosen root, the stabilizer is therefore the product over $i$ of
$\Aut(F_i)_{r_i}\wr\operatorname{Sym}(m_i)$ in its product action.
A stabilizer orbit is specified by a multiset of $m_i$ rooted factor-orbit
labels in each block.  Conversely, equal multisets can be aligned by a
coordinate permutation and then matched by root-fixing factor
automorphisms.  This proves the orbit classification.  The number of
multisets in block $i$ is $\binom{t_i+m_i-1}{m_i}$.  Subtracting the unique
root orbit gives \eqref{eq:orbit-dimension}.

Cartesian distance is the sum of coordinate distances.  A multiset with
multiplicities $\alpha_O$ contributes distance
$\sum_O\alpha_O\rho(O)$, so multiplication over blocks gives
\eqref{eq:orbit-profile}.  Vertex-transitivity permits each arbitrary
root coordinate to be carried to the chosen representative $r_i$, proving
root-independence.  This choice also avoids assuming a canonical
identification between rooted orbit labels at different vertices.
\end{proof}

\subsection{Metric loss and orbit splitting}
Suppose in addition that the factors in \cref{prop:orbit-product} are
distance-regular, with diameters $D_i$.  Let $Q_a^-$ be the span of the
indicators of the non-root cells specified, in each block, by the
multiset of coordinate distances.  This is the symmetric distance
quotient, as a subspace of the residual vertex space.

\begin{proposition}[The two sources of defect]\label{prop:split}
Under these hypotheses,
\begin{equation}\label{eq:sandwich}
 P_a\subseteq Q_a^-\subseteq U_a^-,\qquad
 \dim Q_a^-+1=\prod_i\binom{D_i+m_i}{m_i}.
\end{equation}
Consequently
\begin{equation}\label{eq:defect-split}
 \defect(G,a)=
 \underbrace{\dim Q_a^--\dim P_a}_{\delta_{\mathrm{met}}\ge0}
 +\underbrace{\dim U_a^--\dim Q_a^-}_{\delta_{\mathrm{orb}}\ge0}.
\end{equation}
Here $\dim P_a$ is given exactly by \cref{thm:main} and
$\dim U_a^-$ by \cref{prop:orbit-product}.
\end{proposition}
\begin{proof}
The tensor product of the factor sphere spans is invariant under
adjacency and the global first-sphere projection.  Permuting equal
coordinates commutes with these operators and fixes the all-ones seed.
Thus $P_a$ lies in the permutation-fixed part of this tensor space after
deletion, which is precisely $Q_a^-$.  A stabilizer orbit type refines a
coordinate-distance multiset, by \cref{prop:orbit-product}.  Therefore
each distance-multiset cell is a union of stabilizer orbits, proving
$Q_a^-\subseteq U_a^-$.  The dimension count selects, independently in
each block, a multiset of size $m_i$ from $D_i+1$ distances, and removes
the root cell.  The split follows from the inclusions.
\end{proof}

\begin{remark}[Why the hypotheses differ]
The inclusion $P_a\subseteq Q_a^-$ and its dimension upper bound can be
formed for identical rooted equitable factors without primeness.
However, $Q_a^-\subseteq U_a^-$ need not hold for an arbitrary nonprime
presentation.  For example, presenting $Q_3$ as $Q_2\square K_2$ gives a
five-dimensional residual distance-vector space, whereas its full root
stabilizer has only three non-root orbits.  The prime-factor hypotheses
in \cref{prop:split} cannot be dropped by simply invoking a blockwise
wreath product.
\end{remark}

For a distance-regular graph itself, define
\begin{equation}\label{eq:graded-defect}
 \Delta_{G,a}(z)=\Omega_{G,a}(z)-(1+z+\cdots+z^D).
\end{equation}
By \cref{prop:dr-module}, this has nonnegative coefficients and
$\Delta_{G,a}(1)=\defect(G,a)$.  Its coefficient of $z^j$ is the number
of stabilizer orbits in $S_j(a)$ minus one.  We do not assign this formula
to the principal module of a general product: that module need not be
the direct sum of one line in each total-distance sphere.

If a vertex-transitive distance-regular graph has zero defect, its root
stabilizer is transitive on every distance sphere.  Together with
vertex-transitivity, this is equivalent to distance-transitivity.
For instance $\defect(H(d,q),a)=0$: alphabet permutations fixing root
symbols and coordinate permutations are transitive on each sphere.
Likewise $\defect(J(v,k),a)=0$, for $1\le k\le v/2$, because the subgroup
$\operatorname{Sym}(k)\times\operatorname{Sym}(v-k)$ fixing a root
$k$-subset is transitive on $k$-subsets with a given intersection size
with the root.  These are consequences of their familiar permutation
actions, not new distance-transitivity results.

\subsection{Shrikhande factors and Doob graphs}
We give the small rooted calculation underlying the Doob formulas.
Write $\mathbb Z_4^2$ additively and let
\[
 S=\{\pm(1,0),\ \pm(0,1),\ \pm(1,1)\}.
\]
The Shrikhande graph $\Sh$ is the Cayley graph on $\mathbb Z_4^2$ with
connection set $S$ and root $o=(0,0)$.

\begin{lemma}\label{lem:shrikhande}
The graph $\Sh$ is vertex-transitive, Cartesian-prime, and strongly
regular with parameters $(16,6,2,2)$.  Its rooted stabilizer orbits are
\[
 \{o\},\quad S,\quad
 T_3=\{(0,2),(2,0),(2,2)\},\quad
 T_6=\{(1,2),(1,3),(2,1),(2,3),(3,1),(3,2)\}.
\]
Their sizes are $1,6,3,6$, at distances $0,1,2,2$.  The local graph is
$C_6$, and the rooted spectral generating function is
\begin{equation}\label{eq:shrikhande-Phi}
 \Phi_{\Sh}(z)=\frac{e^{6z}+6e^{2z}+9e^{-2z}}{16}
              =\Phi_{K_4}(z)^2.
\end{equation}
Consequently $\Sh$ and $K_4$ belong to the same boundary class.
\end{lemma}
\begin{proof}
Translations give vertex-transitivity and the connection set gives
valency six.  The linear map
$\tau(x,y)=(y,-x+y)$ preserves $S$.  Its orbits on the nonzero elements
are exactly $S,T_3,T_6$: it acts as a six-cycle on $S$, a three-cycle on
$T_3$, and a six-cycle on $T_6$.  For the representatives $(1,0),(0,2)$,
and $(1,2)$, the common-neighbor sets with $o$ are respectively
\[
 \{(0,3),(1,1)\},\qquad
 \{(0,1),(0,3)\},\qquad
 \{(0,1),(1,1)\}.
\]
Translation and $\tau$ therefore show that every pair of distinct
vertices has two common neighbors.  This proves the strongly regular
parameters and diameter two.  The induced graph on $S$ is the cycle
\[
 (1,0),(1,1),(0,1),(3,0),(3,3),(0,3),(1,0),
\]
as is checked by subtracting consecutive elements and the other pairs.
For $T_3$ the two neighbors in $S$ are nonadjacent in this cycle, whereas
for $T_6$ they are adjacent.  This property is preserved by every
root-fixing automorphism, so the two sets cannot merge.  Since $\tau$
is transitive on each displayed set, these are exactly the stabilizer
orbits.

A nontrivial Cartesian product of connected graphs has a disconnected
local graph: neighbors belonging to different coordinate directions
are never adjacent, and every direction is nonempty.  The connected
local graph $C_6$ therefore proves Cartesian primality of $\Sh$.

The strongly regular identity is $A^2=4I+2J$, so the two nonprincipal
eigenvalues are $2$ and $-2$.  Their multiplicities are $6$ and $9$, from
the total dimension and $\tr A=0$.  Vertex-transitivity makes every
diagonal entry of a spectral projection equal to its rank divided by
$16$.  This gives the first expression in \eqref{eq:shrikhande-Phi}.
Since $\Phi_{K_4}(z)=(e^{3z}+3e^{-z})/4$, the second expression follows.
The valencies are six and three, so their normalized logarithmic
derivatives are equal.
\end{proof}

Define the Doob graph
\[
 \operatorname{Doob}(m,n)=\Sh^{\square m}\square K_4^{\square n},
 \qquad m,n\ge0,\quad N=2m+n\ge1.
\]
The convention allows a missing block when its exponent is zero.

\begin{theorem}[Doob defects]\label{thm:doob}
The graph $\operatorname{Doob}(m,n)$ is distance-regular of diameter
$N=2m+n$ with intersection numbers
\begin{equation}\label{eq:doob-array}
 c_j=j,\qquad a_j=2j,\qquad b_j=3(N-j).
\end{equation}
At every root,
\begin{align}
 \dim P_a&=N,\qquad
 \dim U_a^-=(n+1)\binom{m+3}{3}-1,\label{eq:doob-dim}\\
 \defect&=(n+1)\binom{m+3}{3}-1-N.\label{eq:doob-defect}
\end{align}
Its orbit profile and graded defect are
\begin{align}
 \Omega_{m,n}(z)
 &=\left(\sum_{\alpha_0+\alpha_1+\alpha_2+\alpha_3=m}
 z^{\alpha_1+2\alpha_2+2\alpha_3}\right)(1+z+\cdots+z^n),
 \label{eq:doob-profile}\\
 \Delta_{m,n}(z)&=\Omega_{m,n}(z)-(1+z+\cdots+z^N),
 \label{eq:doob-Delta}
\end{align}
where the summation indices are nonnegative integers.  In particular,
\begin{equation}\label{eq:top-defect}
 [z^N]\Delta_{m,n}(z)=m.
\end{equation}
The split from \cref{prop:split} is
\begin{align*}
 \delta_{\mathrm{met}}&=(n+1)\binom{m+2}{2}-1-N,\\
 \delta_{\mathrm{orb}}&=(n+1)\left(\binom{m+3}{3}-\binom{m+2}{2}\right).
\end{align*}
\end{theorem}
\begin{proof}
A Shrikhande coordinate at distance $s\in\{0,1,2\}$ has $s$ neighbors at
distance $s-1$, $2s$ at distance $s$, and $6-3s$ at distance $s+1$.
A $K_4$ coordinate, for $s\in\{0,1\}$, has the corresponding counts
$s,2s,3-3s$.  At a product vertex of total distance $j$, summing these
counts gives exactly \eqref{eq:doob-array}, independently of its
coordinate-distance pattern.  Distances add, so the diameter is $N$.
This proves distance-regularity and gives $\dim P_a=N$ by
\cref{prop:dr-module}.  Equivalently, \cref{lem:shrikhande} puts every
factor in one boundary class; its total spectral support is
$\{3N-4j:0\le j\le N\}$, and \cref{thm:main} gives the same dimension.

The two prime factors are nonisomorphic and vertex-transitive.  Their
numbers of rooted orbits are four and two.  Formula
\eqref{eq:orbit-dimension} gives \eqref{eq:doob-dim}, hence
\eqref{eq:doob-defect}.  Their rooted orbit distances are $(0,1,2,2)$
and $(0,1)$, so \eqref{eq:orbit-profile} gives
\eqref{eq:doob-profile}, including zero exponents by the empty-product
convention.  Equation \eqref{eq:graded-defect} gives
\eqref{eq:doob-Delta}.

To obtain degree $N=2m+n$, every Shrikhande coordinate must have distance
two and every $K_4$ coordinate distance one.  The first block then has
$\alpha_0=\alpha_1=0$ and $\alpha_2+\alpha_3=m$, giving $m+1$ orbit
types.  Subtracting the single primary line proves
\eqref{eq:top-defect}.  Finally \eqref{eq:sandwich}, with factor diameters
two and one, gives
$\dim Q_a^-=(n+1)\binom{m+2}{2}-1$, and hence the split.
\end{proof}

The intersection array \eqref{eq:doob-array} is that of $H(N,4)$.
Egawa's classification~\cite{Egawa1981} says that distance-regular graphs
with this array are exactly the corresponding Doob graphs, including
the Hamming case $m=0$.  Therefore, within this classified family, the
array together with the coefficient in \eqref{eq:top-defect} specifies
$m$ and $n=N-2m$, and hence the isomorphism type.  The classification is
used only for this final interpretation; none of the preceding product
or defect formulas depends on it.

\subsection{What principal moments do not detect}
For fixed $N$, all Doob graphs in \cref{thm:doob} have the same principal
adjacency--degree moments, by \cref{thm:moments}.  Their defects need not
be equal.  This proves that the defect is not determined by those
principal moments.  It does not establish a refinement relation between
the defect and any entire family of full-matrix spectral invariants.

There is also a direct punctured trace that separates these graphs.
The operator $EBE$ is adjacency of the local graph, extended by zero.
At a Doob root this local graph is $mC_6\sqcup nK_3$, whereas at a root of
$H(N,4)$ it is $NK_3$.  Since a triangle contributes six to the cubic
adjacency trace and a six-cycle contributes zero,
\begin{equation}\label{eq:local-cubic}
 \tr((EBE)^3)=6n=6(N-2m).
\end{equation}
Thus the difference from the Hamming value is $12m$.  The word in
\eqref{eq:local-cubic} is available from the punctured generators because
$E=3NI-D(G-a)$.  This distinguishes full traces from principal moments:
on the primary punctured module, $EBE=2E$ and its cubic trace is always
$8$, whereas the full trace also sees nonprimary local constituents.

\section{Further questions}\label{sec:questions}

The product theorem reduces the principal-module problem to two pieces
of finite data: the partition by normalized logarithmic derivatives and
the affine ranks of additive fibres.  Several questions remain after
that reduction.  One is to classify pairs of distance-regular graphs
satisfying \eqref{eq:Phi-power}.  The pair $\Sh,K_4$ shows that equal
boundary classes need not have equal diameters or equal spectra.  A
structural description of such pairs would complement the exact
spectral test.

For three or more boundary classes, the fibres have dimension at most
$t-1$ but can have further affine dependencies.  The two-class formula
in \cref{thm:complete} is a count of consecutive triples in a rectangle.
An analogous combinatorial evaluation for several distinct complete
factors would make the higher-rank cases equally explicit.  The theorem
already computes these dimensions exactly; what is missing is a useful
closed evaluation of the resulting finite ranks for broad families.

The full-algebra problem is separate.  By \cref{lem:corner}, for
connected deletions it asks when the first-sphere projection and
residual adjacency recover all distance projections.  The Hamming
calculations show precisely how later starting levels can obstruct this.
Other distance-regular families may exhibit both endpoint loss and
identification of distinct irreducible modules.  The general product
classification of the principal module does not, by itself, determine
these full compressed algebras.

\appendix
\section{Finite consistency checks}\label{app:checks}

The proofs above do not depend on computation.  The following checks
provide reproducible tests of the formulas and of the distinction
between principal-module dimension and full-algebra dimension.
For a hypercube, generate the span of all words in its residual matrices
$B,E$, starting at $I$.  Computation over $\mathbb F_{1000003}$ gives
\begin{center}
\begin{tabular}{c*{6}{r}}
\toprule
$d$&1&2&3&4&5&6\\
\midrule
Computed word-span rank modulo $1000003$&1&5&13&26&43&64\\
$\dim\Balg_a$ from \cref{thm:cube}&1&5&13&26&43&64\\
$\dim e\Talg(a)e$&1&5&13&26&45&71\\
\bottomrule
\end{tabular}
\end{center}
A modular rank is a lower bound for the corresponding rational matrix
rank, not a stand-alone proof of equality over $\R$.  The analytic proof
of \cref{thm:cube} supplies the matching upper bound for every $d$.

For $q\ge3$, the full punctured Hamming algebras were checked in the
same manner for $(d,q)=(1,3),(2,3),(2,4),(3,3),(3,4),(4,3)$.
The resulting ranks are $2,10,10,24,24,44$, respectively, in agreement
with \cref{thm:hamming}; the corresponding compressed-corner dimensions
are $2,10,10,28,28,61$.

For mixed complete powers, one can instead work on the much smaller
quotient indexed by $(i,j)\in\{0,\ldots,m\}\times\{0,\ldots,n\}$ with
$(0,0)$ removed.  On cell-constant functions, adjacency has diagonal
$i(p-2)+j(q-2)$, and row entries to the cells
$(i-1,j),(i+1,j),(i,j-1),(i,j+1)$ equal
\[
 i,\quad(m-i)(p-1),\quad j,\quad(n-j)(q-1),
\]
respectively, when the destination cell is present.  The projection $E$
is diagonal with entry $\mathbf1_{i+j=1}$.  Closing the all-ones vector
under these two integer matrices gives the following ranks, agreeing
with \cref{thm:complete}:
\begin{center}
\begin{tabular}{rrrrrrr}
\toprule
$p$&$q$&$m$&$n$&$\dim P_a$&$\dim U_a^-$&$\defect$\\
\midrule
2&3&3&2&11&11&0\\
2&3&6&4&33&34&1\\
2&3&8&5&47&53&6\\
3&6&4&2&13&14&1\\
3&6&6&4&25&34&9\\
4&6&6&4&33&34&1\\
3&4&4&3&19&19&0\\
\bottomrule
\end{tabular}
\end{center}
These are finite checks, not substitutes for the general fibre-counting
argument.  The accompanying verification script records the field used
for every modular computation; it is archived, together with a Lean~4
formalization of the combinatorial core of \cref{thm:complete}, at
\url{https://github.com/Lzp88/punctured-ad-verification}.  All numerical
statements in this appendix are consistency checks; none is used to
establish a general theorem.

\section*{Declaration on the use of generative AI}

The adjacency--degree algebra program originates in the author's joint
work with Pengxiang Li~\cite{LuLi2026}.  The punctured variant studied
here grew out of the author's discussions with OpenAI's GPT on how that
framework should treat regular graphs, and the results and proofs of
this paper were developed and brought to their present form in sustained
interaction with that model, under the author's direction.

The manuscript was independently audited with Anthropic's Claude.  The
audit comprised a line-by-line mathematical review of all proofs, the
exact rational and modular consistency computations reported in
\cref{app:checks}, and a machine-checked formalization in Lean~4 over
Mathlib of the following components: the lattice fibre-counting
identity and the resulting closed form \eqref{eq:complete-dim} in the
proof of \cref{thm:complete}; the rank evaluation $\min\{2,r\}$ for the
two-class fibre matrices; the closed form of $\eta_q$ displayed in that
proof, together with the fact that distinct complete graphs lie in
distinct boundary classes; and the inequality part of \cref{lem:sumset}.
The formalization therefore machine-verifies the derivation of
\cref{thm:complete} from \cref{thm:main}; \cref{thm:main} itself and the
results of Sections~\ref{sec:dr}--\ref{sec:orbits} were verified by the
line-by-line review and the finite computations only.  The Lean
theorems contain no unproven assumptions and depend on no axioms beyond
the standard Lean/Mathlib base.

The verification script, its exact JSON certificate, and the complete
Lean development, with pinned toolchain and dependency versions, are
publicly archived at
\url{https://github.com/Lzp88/punctured-ad-verification}.

The author directed and reviewed all of this work and takes full
responsibility for the content of the paper.

\end{document}